\documentclass[11pt]{article}
\usepackage{amsmath,amsthm,amssymb}
\usepackage[dvips]{graphicx}

\newcommand{\R}{\mathbb R}
\newcommand{\C}{\mathbb C}

\renewcommand{\P}{\mathbb P}
\newcommand{\CAP}{\operatorname{cap}}
\newcommand{\ii}{\operatorname{i}}
\newcommand{\sn}{\operatorname{sn}}
\newcommand{\cn}{\operatorname{cn}}
\newcommand{\dn}{\operatorname{dn}}
\newcommand{\zn}{\operatorname{zn}}
\newcommand{\cc}{^{\operatorname{c}}}

\begin{document}


\title{Estimates for the Asymptotic Convergence Factor of Two Intervals\footnote{published in: Journal of Computational and Applied Mathematics {\bf 236} (2011), 26--36.}}
\author{Klaus Schiefermayr\footnote{University of Applied Sciences Upper Austria, School of Engineering and Environmental Sciences, Stelzhamerstrasse\,23, 4600 Wels, Austria, \textsc{klaus.schiefermayr@fh-wels.at}}}
\date{}
\maketitle

\theoremstyle{plain}
\newtheorem{theorem}{Theorem}
\newtheorem{corollary}{Corollary}
\newtheorem{lemma}{Lemma}
\newtheorem{definition}{Definition}
\theoremstyle{definition}
\newtheorem*{remark}{Remark}
\newtheorem*{example}{Example}

\begin{abstract}
Let $E$ be the union of two real intervals not containing zero. Then $L_n^r(E)$ denotes the supremum norm of that polynomial $P_n$ of degree less than or equal to $n$, which is minimal with respect to the supremum norm provided that $P_n(0)=1$. It is well known that the limit $\kappa(E):=\lim_{n\to\infty}\sqrt[n]{L_n^r(E)}$ exists, where $\kappa(E)$ is called the asymptotic convergence factor, since it plays a crucial role for certain iterative methods solving large-scale matrix problems. The factor $\kappa(E)$ can be expressed with the help of Jacobi's elliptic and theta functions, where this representation is very involved. In this paper, we give precise upper and lower bounds for $\kappa(E)$ in terms of elementary functions of the endpoints of $E$.
\end{abstract}



\section{Introduction}


For $n\in{\mathbb N}$, let $\P_n$ denote the set of all polynomials of degree at most $n$ with real coefficients. Let $E$ be the union of two real intervals, i.e.,
\begin{equation}\label{E}
E:=[a_1,a_2]\cup[a_3,a_4],\qquad~a_1<a_2<a_3<a_4,
\end{equation}
and let the supremum norm $\|\cdot\|_E$ associated with $E$ be defined by
\begin{equation}
\|P_n\|_E:=\max_{x\in{E}}|P_n(x)|
\end{equation}
for any polynomial $P_n\in\P_n$. Consider the following two classical approximation problems:
\begin{equation}\label{Ln}
L_n(E):=\|T_n(\cdot,E)\|_E:=\min\bigl\{\|P_n\|_E:P_n\in\P_n\setminus\P_{n-1},P_n~\text{monic~polynomial}\bigr\}
\end{equation}
and, $0\notin{E}$,
\begin{equation}\label{Lnr}
L_n^r(E,0):=\|R_n(\cdot,E,0)\|_E:=\min\bigl\{\|P_n\|_E:P_n\in\P_n,P_n(0)=1\bigr\}.
\end{equation}


The optimal (monic) polynomial $T_n(x,E)=x^n+\ldots\in\P_n\setminus\P_{n-1}$ in \eqref{Ln} is called the Chebyshev polynomial on $E$ and $L_n(E)$ is called the minimum deviation of $T_n(\cdot,E)$ on $E$. It is well known that the limit
\begin{equation}\label{cap}
\CAP{E}:=\lim_{n\to\infty}\sqrt[n]{L_n(E)}
\end{equation}
exists, where $\CAP{E}$ is called the Chebyshev constant or the logarithmic capacity of $E$. Concerning the general properties of $\CAP{C}$, $C\subset\C$ compact, we refer to \cite{Kirsch} and \cite[chapter\,5]{Ransford}.


The optimal polynomial $R_n(\cdot,E,0)\in\P_n$ in \eqref{Lnr} is called the \emph{minimal residual polynomial} for the degree $n$ on $E$ and the quantity $L_n^r(E,0)$ is called the minimum deviation of $R_n(\cdot,E,0)$ on $E$. Note that we say \emph{for} the degree $n$ but not \emph{of} degree $n$ since the minimal residual polynomial for the degree $n$ on $E$ is a polynomial of degree $n$ or $n-1$, see \cite{Sch-2010}. As above, the limit
\begin{equation}\label{kappa}
\kappa(E,0):=\lim_{n\to\infty}\sqrt[n]{L_n^r(E,0)}
\end{equation}
exists, see, e.g.\ \cite{Kuijlaars} or \cite{DTT}, where $\kappa(E,0)$ is usually called the \emph{estimated asymptotic convergence factor}. The approximation problem \eqref{Lnr} and the convergence factor \eqref{kappa} arise for instance in the context of solving large-scale matrix problems by Krylov subspace iterations. There is an enormous literature on these subject, hence we would like to mention only three references, the review of Discroll, Toh and Trefethen\,\cite{DTT}, the book of Fischer\,\cite{Fischer-Book} and the review of Kuijlaars\,\cite{Kuijlaars}.


In the case of two intervals, both terms, $\kappa(E,0)$ and $\CAP{E}$, can be expressed with the help of Jacobi's elliptic and theta functions and this characterization goes back to the work of Achieser\,\cite{Achieser-1932}. Since, in both cases, the representation is very involved, it is desirable to have at least estimates of a simpler form. For $\CAP{E}$, such estimates are given in \cite{Solynin}, \cite{Sch-2008-1}, and \cite{DubininKarp}. In this paper, we will give a precise upper and lower bound for $\kappa(E,0)$ in terms of elementary functions of the endpoints $a_1,a_2,a_3,a_4$ of $E$.


The paper is organized as follows. In Section\,2, we recall the representations of $\kappa(E,0)$ and $\CAP{E}$ with the help of Jacobi's elliptic and theta functions. Using an inequality between a Jacobian theta function and the Jacobian elliptic functions, proved in Section\,6, we obtain an upper and a lower bound for $\kappa(E,0)$ in Section\,3, which is the main result of the paper. In Section\,4, the following extremum problem is solved: Given the length of the two intervals and the length of the gap between the two intervals, for which set of two intervals the convergence factor $\kappa(E,0)$ gets minimal? In Section\,5, as a byproduct, a new and simple lower bound for $\CAP{E}$ is derived. Finally, in Section\,6, the notion of Jacobi's elliptic and theta functions is recapitulated and several new inequalities, needed in Section\,3 and 4, are proved.

\section{Representation of the Asymptotic Factor and the Logarithmic Capacity in Terms of Jacobi's Elliptic Functions}


Let $E$ be given as in \eqref{E} such that $0\notin{E}$. It is convenient to use the linear transformation
\begin{equation}\label{ell}
\ell(x):=\frac{2x-a_1-a_4}{a_4-a_1},
\end{equation}
which maps the set $E$ onto the normed set
\begin{equation}\label{Eh}
\hat{E}:=[-1,\alpha]\cup[\beta,1],
\end{equation}
where $\alpha:=\ell(a_2)$ and $\beta:=\ell(a_3)$. For the corresponding Chebyshev polynomials, we have
\begin{equation}
T_n(x,E)=\Bigl(\frac{a_4-a_1}{2}\Bigr)^nT_n(\ell(x),\hat{E}),
\end{equation}
thus
\begin{equation}
L_n(E)=\Bigl(\frac{a_4-a_1}{2}\Bigr)^nL_n(\hat{E})
\end{equation}
and
\begin{equation}
\CAP{E}=\frac{a_4-a_1}{2}\,\CAP\hat{E}.
\end{equation}
Concerning the minimal residual polynomial, there is
\begin{equation}
R_n(x,E,0)=R_n(\ell(x),\hat{E},\xi),
\end{equation}
where $\xi:=\ell(0)$, thus
\begin{equation}
L_n^r(E,0)=L_n^r(\hat{E},\xi)
\end{equation}
and
\begin{equation}
\kappa(E,0)=\kappa(\hat{E},\xi),
\end{equation}
for details, see \cite[Sec.\,3.2]{Fischer-Book}.


Let $\hat{E}$ be given as in \eqref{Eh} with $-1<\alpha<\beta<1$ and let $\xi\in\R\setminus\hat{E}$. Then there exists a (uniquely determined) Green's function for $\hat{E}\cc:=\overline{\C}\setminus\hat{E}$ (where $\overline{\C}:=\C\cup\infty$) with pole at infinity, denoted by $g(z;\hat{E}\cc,\infty)$. The Green's function is defined by the following three properties:
\begin{itemize}
\item $g(z;\hat{E}\cc,\infty)$ is harmonic in $\hat{E}\cc$.
\item $g(z;\hat{E}\cc,\infty)-\log|z|$ is harmonic in a neighbourhood of infinity.
\item $g(z;\hat{E}\cc,\infty)\to0$ as $z\to\hat{E}$, $z\in\hat{E}\cc$.
\end{itemize}
With the Green's function $g(z;\hat{E}\cc,\infty)$, the estimated asymptotic convergence factor\\ $\kappa(\hat{E},\xi)$ can be characterized by
\begin{equation}\label{kappa-g}
\kappa(\hat{E},\xi)=\exp(-g(\xi;\hat{E}\cc,\infty)).
\end{equation}
This connection was first observed by Eiermann, Li and Varga\,\cite{ELV} (for more general sets), see also \cite[Sec.\,3.1]{Fischer-Book}, \cite{Kuijlaars} and \cite{DTT}.


Let us recall the construction of the Green's function for $\hat{E}\cc$, due to Achieser\,\cite{Achieser-1932}, see also \cite{Fischer-1992} and in particular \cite[Chapter\,3]{Fischer-Book}. This characterization is mainly based on a heavy usage of Jacobi's elliptic and theta functions. For the notation and some basic properties of this class of functions, see the beginning of Section\,6.


Define the modulus $k$ of Jacobi's elliptic functions $\sn(u)$, $\cn(u)$ and $\dn(u)$ and of Jacobi's theta functions $\Theta(u)$, $H(u)$, $H_1(u)$ and $\Theta_1(u)$ by
\begin{equation}\label{k}
k=\sqrt{\frac{2(\beta-\alpha)}{(1-\alpha)(1+\beta)}}.
\end{equation}
Then the complementary modulus $k':=\sqrt{1-k^2}$ is given by
\begin{equation}\label{k'}
{k'}=\sqrt{\frac{(1+\alpha)(1-\beta)}{(1-\alpha)(1+\beta)}}.
\end{equation}
Note that $0<k,k'<1$. Let $K\equiv{K}(k)$ be the complete elliptic integral of the first kind and let $K'\equiv{K}'(k):=K(k')$. Let $0<\rho<K$ be uniquely defined by the equation
\begin{equation}\label{sn}
\sn^2(\rho)=\frac{1-\alpha}{2}.
\end{equation}
By \eqref{k}, \eqref{sn} and \eqref{sncndn},
\begin{equation}\label{cndn}
\cn^2(\rho)=\frac{1+\alpha}{2} \qquad\text{and}\qquad\dn^2(\rho)=\frac{1+\alpha}{1+\beta}.
\end{equation}
Further, consider the function
\begin{equation}\label{phi}
\varphi(u):=\frac{\sn^2(u)\cn^2(\rho)+\cn^2(u)\sn^2(\rho)}{\sn^2(u)-\sn^2(\rho)}.
\end{equation}
Let
\[
{\cal{P}}:=\bigl\{u\in\C:u=\lambda{K}+\ii\lambda'K',~0<\lambda<1,~-1<\lambda'\leq1\bigr\}
\]
then $\varphi:{\cal P}\to\hat{E}\cc$ is a bijective mapping and especially the mappings $\varphi:[0,\rho)\to(-\infty,-1]$, $\varphi:[\rho,K]\to[1,\infty)$ and $\varphi:[\ii{K}',K+\ii{K}']\to[\alpha,\beta]$ are bijective.


Then the Green's function for $\hat{E}\cc$ is given by
\begin{equation}
g(z;\hat{E}\cc,\infty)=\log\Bigl|\frac{H(u+\rho)}{H(u-\rho)}\Bigr|,\quad\text{where}\quad z=\varphi(u).
\end{equation}
Since $\xi\in\R\setminus\hat{E}$, $u^*\in(0,K)\cup(\ii{K}',K+\ii{K}')$ is uniquely determined by the equation $\varphi(u^*)=\xi$. Thus, by \eqref{kappa-g}, the convergence factor $\kappa(\hat{E},\xi)$ can be computed by
\[
\kappa(\hat{E},\xi)=\Bigl|\frac{H(u^*-\rho)}{H(u^*+\rho)}\Bigr|.
\]
Let us summarize these results in the following theorem.


\begin{theorem}[Fischer\,\cite{Fischer-Book},~Achieser\,\cite{Achieser-1932}]\label{Thm-kappa}
Let $\hat{E}:=[-1,\alpha]\cup[\beta,1]$, $-1<\alpha<\beta<1$, let $\xi\in\R\setminus\hat{E}$, and let $k\in(0,1)$ and $\rho\in(0,K)$ be given by \eqref{k} and \eqref{sn}, respectively. Then, the asymptotic convergence factor $\kappa(\hat{E},\xi)$ is given by
\begin{equation}
\kappa(\hat{E},\xi)=\Bigl|\frac{H(u^*-\rho)}{H(u^*+\rho)}\Bigr|,
\end{equation}
where $u^*\in(0,K)\cup(\ii{K}',K+\ii{K}')$ is uniquely determined by the equation $\varphi(u^*)=\xi$, $\varphi$ defined in \eqref{phi}.
\end{theorem}


On the other hand, concerning the logarithmic capacity of $\hat{E}$, Achieser\,\cite{Achieser-1930} proved the following, see also \cite[Cor.\,8]{PehSch-2004}.


\begin{theorem}[Achieser\,\cite{Achieser-1930}]\label{Thm-AchieserCap}
Let $\hat{E}:=[-1,\alpha]\cup[\beta,1]$, $-1<\alpha<\beta<1$, and let $k\in(0,1)$ and $\rho\in(0,K)$ be given by \eqref{k} and \eqref{sn}, respectively. Then, the logarithmic capacity of $\hat{E}$ is given by
\begin{equation}\label{cap1}
\CAP\hat{E}=\frac{1+\beta}{2(1+\alpha)}\cdot\frac{\Theta^4(0)}{\Theta^4(\rho)}.
\end{equation}
\end{theorem}

\section{Bounds for the Asymptotic Covergence Factor of Two Intervals}


\begin{theorem}\label{Thm-ineq-kappa}
Let $\hat{E}:=[-1,\alpha]\cup[\beta,1]$, $-1<\alpha<\beta<1$ and let $\xi\in\R\setminus\hat{E}$. Then, for the convergence factor $\kappa(\hat{E},\xi)$, the inequalities
\begin{equation}\label{ineq-kappa}
\frac{A_2}{A_1}\cdot{B}\leq\kappa(\hat{E},\xi)\leq\frac{A_1}{A_2}\cdot{B}
\end{equation}
hold, where
\begin{equation}\label{A}
\begin{aligned}
A_1&:=\sqrt[4]{(1-\alpha)(1+\beta)}+\sqrt[4]{(1+\alpha)(1-\beta)},\\
A_2&:=\sqrt[4]{8}\sqrt[4]{\sqrt{(1-\alpha)(1+\beta)}+\sqrt{(1+\alpha)(1-\beta)}}\sqrt[16]{(1-\alpha^2)(1-\beta^2)},
\end{aligned}
\end{equation}
and $B$ is given in the following:
\begin{enumerate}
\item For $\alpha<\xi<\beta$,
\begin{equation}\label{B1}
B:=\frac{\sqrt[4]{(1+\alpha)(1-\beta)}+\sqrt{1-\xi}-\sqrt{(\xi-\alpha)(\beta-\xi)}}
{\sqrt[4]{(1+\alpha)(1-\beta)}+\sqrt{1-\xi}+\sqrt{(\xi-\alpha)(\beta-\xi)}}.
\end{equation}
\item For $\xi\in\R\setminus[-1,1]$,
\begin{equation}\label{B2}
\begin{aligned}
B:=&\frac{(2\xi-\xi\alpha+\xi\beta-\alpha-\beta)\sqrt[4]{\frac{(1+\alpha)(1-\beta)}{(1-\alpha)(1+\beta)}}
+2\sqrt{(\xi-\alpha)(\xi-\beta)}-(\beta-\alpha)\sqrt{\xi^2-1}}
{(2\xi-\xi\alpha+\xi\beta-\alpha-\beta)\sqrt[4]{\frac{(1+\alpha)(1-\beta)}{(1-\alpha)(1+\beta)}}
+2\sqrt{(\xi-\alpha)(\xi-\beta)}+(\beta-\alpha)\sqrt{\xi^2-1}}\\
&\times\frac{\bigl|\sqrt{(1+\xi)(\xi-\alpha)}-\sqrt{(\xi-1)(\xi-\beta)}\bigr|}
{\sqrt{(1+\xi)(\xi-\alpha)}+\sqrt{(\xi-1)(\xi-\beta)}}
\end{aligned}
\end{equation}
\end{enumerate}
\end{theorem}


\begin{proof}
By \eqref{sn}, \eqref{cndn} and \eqref{sncndn}, the mapping $\varphi(u)$ in \eqref{phi} may be rewritten as
\begin{equation}\label{phi-1}
\varphi(u)=\alpha+\frac{1-\alpha^2}{2\sn^2(u)+\alpha^2-1}
\end{equation}
Let $u^*\in(0,K)\cup(\ii{K}',K+\ii{K}')$ be uniquely determined by the equation $\varphi(u^*)=\xi$. Note that
\begin{equation}\label{u*}
\begin{aligned}
\alpha<\xi<\beta&\iff{u}^*\in(\ii{K}',K+\ii{K}')\\
\xi\in(-\infty,-1)\cup(1,\infty)&\iff{u}^*\in(0,K)
\end{aligned}
\end{equation}
By \eqref{sn} and \eqref{phi-1}, $\varphi(u^*)=0$ is equivalent to
\begin{equation}\label{snu*}
\sn^2(u^*)=\frac{(1+\xi)(1-\alpha)}{2(\xi-\alpha)}=\frac{1+\xi}{\xi-\alpha}\,\sn^2(\rho).
\end{equation}
By \eqref{k}, \eqref{cndn}, \eqref{snu*} and \eqref{sncndn},
\begin{equation}\label{cnu*}
\cn^2(u^*)=\frac{(\xi-1)(1+\alpha)}{2(\xi-\alpha)}=\frac{\xi-1}{\xi-\alpha}\,\cn^2(\rho)
\end{equation}
and
\begin{equation}\label{dnu*}
\dn^2(u^*)=\frac{(\xi-\beta)(1+\alpha)}{(1+\beta)(\xi-\alpha)}=\frac{\xi-\beta}{\xi-\alpha}\,\dn^2(\rho).
\end{equation}
In order to obtain estimates for $\kappa(\hat{E},\xi)$, we will use the inequality
\begin{equation}\label{Theta}
\frac{\sqrt[4]{8(1+k')}\sqrt[8]{k'}}{1+\sqrt{k'}}
\leq\frac{\Theta(u-\rho)}{\Theta(u+\rho)}\cdot\frac{\sqrt{k'}+\dn(u-\rho)}{\sqrt{k'}+\dn(u+\rho)}\leq
\frac{1+\sqrt{k'}}{\sqrt[4]{8(1+k')}\sqrt[8]{k'}}
\end{equation}
which follows immediately from Lemma\,\ref{Lemma-IneqTheta}. By \eqref{k'}, straightforward computation gives
\begin{equation}\label{factor-k'}
\frac{1+\sqrt{k'}}{\sqrt[4]{8(1+k')}\sqrt[8]{k'}}=\frac{A_1}{A_2},
\end{equation}
where $A_1$ and $A_2$ are defined in \eqref{A}. Further, by
\cite[Eq.\,(123.01)]{BF},
\begin{equation}\label{dndn}
\begin{aligned}
&\frac{\sqrt{k'}+\dn(u+\rho)}{\sqrt{k'}+\dn(u-\rho)}\\
&=\frac{\sqrt{k'}(1-k^2\sn^2(u)\sn^2(\rho))+\dn(u)\dn(\rho)-k^2\sn(u)\sn(\rho)\cn(u)\cn(\rho)}
{\sqrt{k'}(1-k^2\sn^2(u)\sn^2(\rho))+\dn(u)\dn(\rho)+k^2\sn(u)\sn(\rho)\cn(u)\cn(\rho)}
\end{aligned}
\end{equation}
We consider the two cases $\alpha<\xi<\beta$ and $\xi\in\R\setminus[-1,1]$.
\begin{enumerate}
\item[1.] $\alpha<\xi<\beta$.\\
By \eqref{u*}, $u^*=v^*+\ii{K}'$ with $0<v^*<K$. With the formula \cite{AS}
\[
H(u+\ii{K}')=\ii\exp(-\tfrac{\pi{K}'}{4K})\exp(-\tfrac{\ii\pi{u}}{2K})\,\Theta(u),
\]
we get
\begin{equation}\label{kappa-case1}
\begin{aligned}
\kappa(\hat{E},\xi)&=\Bigl|\frac{H(v^*-\rho+\ii{K}')}{H(v^*+\rho+\ii{K}')}\Bigr|
=\Bigl|\frac{\ii\exp(-\tfrac{\pi{K}'}{4K})\exp(-\tfrac{\ii\pi(v^*-\rho)}{2K})\,\Theta(v^*-\rho)}
{\ii\exp(-\tfrac{\pi{K}'}{4K})\exp(-\tfrac{\ii\pi(v^*+\rho)}{2K})\,\Theta(v^*+\rho)}\Bigr|\\
&=|\exp(\tfrac{\ii\pi\rho}{K})|\cdot\Bigl|\frac{\Theta(v^*-\rho)}{\Theta(v^*+\rho)}\Bigr|
=\frac{\Theta(v^*-\rho)}{\Theta(v^*+\rho)}.
\end{aligned}
\end{equation}
Thus, by \eqref{Theta} and \eqref{factor-k'},
\begin{equation}\label{ineq-case1}
\frac{A_2}{A_1}\cdot\frac{\sqrt{k'}+\dn(v^*+\rho)}{\sqrt{k'}+\dn(v^*-\rho)}\leq\kappa(\hat{E},\xi)
\leq\frac{A_1}{A_2}\cdot\frac{\sqrt{k'}+\dn(v^*+\rho)}{\sqrt{k'}+\dn(v^*-\rho)}
\end{equation}
By \cite[Eq.\,(122.07)]{BF}
\[
\sn^2(u^*)=\sn^2(v^*+\ii{K}')=\frac{1}{k^2\sn^2(v^*)}
\]
hence, by \eqref{k} and \eqref{snu*}--\eqref{dnu*}, we obtain the formulae
\begin{align}
\sn^2(v^*)&=\frac{1}{k^2\sn^2(u^*)}=\frac{(\xi-\alpha)(1+\beta)}{(1+\xi)(\beta-\alpha)},\label{snv*}\\
\cn^2(v^*)&=1-\sn^2(v^*)=\frac{(\beta-\xi)(1+\alpha)}{(1+\xi)(\beta-\alpha)},\label{cnv*}\\
\dn^2(v^*)&=1-k^2\sn^2(v^*)=\frac{(1-\xi)(1+\alpha)}{(1+\xi)(1-\alpha)}.\label{dnv*}
\end{align}
Starting from relation \eqref{dndn} with $u=v^*$ and using \eqref{k}--\eqref{cndn} and \eqref{snv*}--\eqref{dnv*}, we obtain
\[
\frac{\sqrt{k'}+\dn(v^*+\rho)}{\sqrt{k'}+\dn(v^*-\rho)}=B,
\]
where $B$ is defined in \eqref{B1}. Hence, inequality \eqref{ineq-kappa} follows by \eqref{ineq-case1}.
\item[2.] $\xi\in\R\setminus[-1,1]$.\\
By \eqref{u*}, $0<u^*<K$. By Theorem\,\ref{Thm-kappa}, \eqref{H-H1-T1} and Lemma\,\ref{Lemma-Theta}\,(i),
\begin{equation}\label{kappa-case2}
\kappa(\hat{E},\xi)=\Bigl|\frac{H(u^*-\rho)}{H(u^*+\rho)}\Bigr|
=\Bigl|\frac{\sn(u^*-\rho)}{\sn(u^*+\rho)}\Bigr|\cdot\frac{\Theta(u^*-\rho)}{\Theta(u^*+\rho)}
\end{equation}
Thus, by \eqref{kappa-case2}, \eqref{Theta} and \eqref{factor-k'},
\begin{equation}\label{ineq-case2}
\frac{A_2}{A_1}\cdot\frac{\sqrt{k'}+\dn(u^*+\rho)}{\sqrt{k'}+\dn(u^*-\rho)}
\cdot\Bigl|\frac{\sn(u^*-\rho)}{\sn(u^*+\rho)}\Bigr|\leq\kappa(\hat{E},\xi)
\leq\frac{A_1}{A_2}\cdot\frac{\sqrt{k'}+\dn(u^*+\rho)}{\sqrt{k'}+\dn(u^*-\rho)}
\cdot\Bigl|\frac{\sn(u^*-\rho)}{\sn(u^*+\rho)}\Bigr|
\end{equation}
By the formulae for $\sn(u+v)$ and $\sn(u-v)$, see \cite[Eq.\,(123.01)]{BF}, together with \eqref{snu*}--\eqref{dnu*}, we get
\begin{equation}\label{snsn}
\begin{aligned}
\Bigl|\frac{\sn(u^*-\rho)}{\sn(u^*+\rho)}\Bigr|
&=\Bigl|\frac{\sn(u^*)\cn(\rho)\dn(\rho)-\sn(\rho)\cn(u^*)\dn(u^*)}
{\sn(u^*)\cn(\rho)\dn(\rho)+\sn(\rho)\cn(u^*)\dn(u^*)}\Bigr|\\
&=\frac{\bigl|\sqrt{(1+\xi)(\xi-\alpha)}-\sqrt{(\xi-1)(\xi-\beta)}\bigr|}
{\sqrt{(1+\xi)(\xi-\alpha)}+\sqrt{(\xi-1)(\xi-\beta)}}
\end{aligned}
\end{equation}
Starting from relation \eqref{dndn} with $u=u^*$ and using \eqref{k}--\eqref{cndn}, \eqref{snu*}--\eqref{dnu*} and \eqref{snsn}, we obtain
\[
\frac{\sqrt{k'}+\dn(u^*+\rho)}{\sqrt{k'}+\dn(u^*-\rho)}
\cdot\Bigl|\frac{\sn(u^*-\rho)}{\sn(u^*+\rho)}\Bigr|=B,
\]
where $B$ is defined in \eqref{B2}. Hence, inequality \eqref{ineq-kappa} follows by \eqref{ineq-case2}.
\end{enumerate}
\end{proof}


\begin{remark}
\begin{enumerate}
\item Let $-1<\alpha<\beta<1$. If $\{\alpha,\beta\}$ changes to $\{-\beta,-\alpha\}$, then, by \eqref{k}, the modulus $k$ does not change, and, by \eqref{sn}, $\rho$ changes to $K-\rho$. Thus, by \eqref{snu*},
\begin{equation}
\kappa([-1,\alpha]\cup[\beta,1],\xi)=\kappa([-1,-\beta]\cup[-\alpha,1],\tilde{\xi}),
\end{equation}
where $\tilde{\xi}$ satisfies the equation
\begin{equation}
\frac{(1+\xi)(1-\alpha)}{2(\xi-\alpha)}=\frac{(1+\tilde{\xi})(1+\beta)}{2(\tilde{\xi}+\beta)}.
\end{equation}
Hence, for the plots introduced in (ii), it remains to consider the case $\alpha\leq-\beta$ only.
\item In order to underline the goodness of the estimates for $\kappa(\hat{E},\xi)$ given in Theorem\,\ref{Thm-ineq-kappa}, let us present some plots, see Fig.\,\ref{Fig_AsymptoticFactor}. For the six cases $\{\alpha,\beta\}=\{-0.2,0.1\}$, $\{\alpha,\beta\}=\{-0.5,0.0\}$, $\{\alpha,\beta\}=\{-0.5,0.5\}$, $\{\alpha,\beta\}=\{-0.9,-0.3\}$, $\{\alpha,\beta\}=\{-0.9,0.5\}$, $\{\alpha,\beta\}=\{-0.9,0.9\}$, we have plotted the graph of $\kappa(\hat{E},\xi)$ (solid line),  the graph of the upper bound in \eqref{ineq-kappa} (dashed line), and the graph of the lower bound in \eqref{ineq-kappa} (dotted line) for $\alpha\leq\xi\leq\beta$. As one can see, the graphs match nearly perfectly, only if the length of the intervals $[-1,\alpha]$ and $[\beta,1]$ is very small, there is a visually recognizable difference between the bounds and the exact value $\kappa(\hat{E},\xi)$.
\end{enumerate}
\end{remark}


\begin{figure}[ht]
\begin{center}
\includegraphics[scale=0.9]{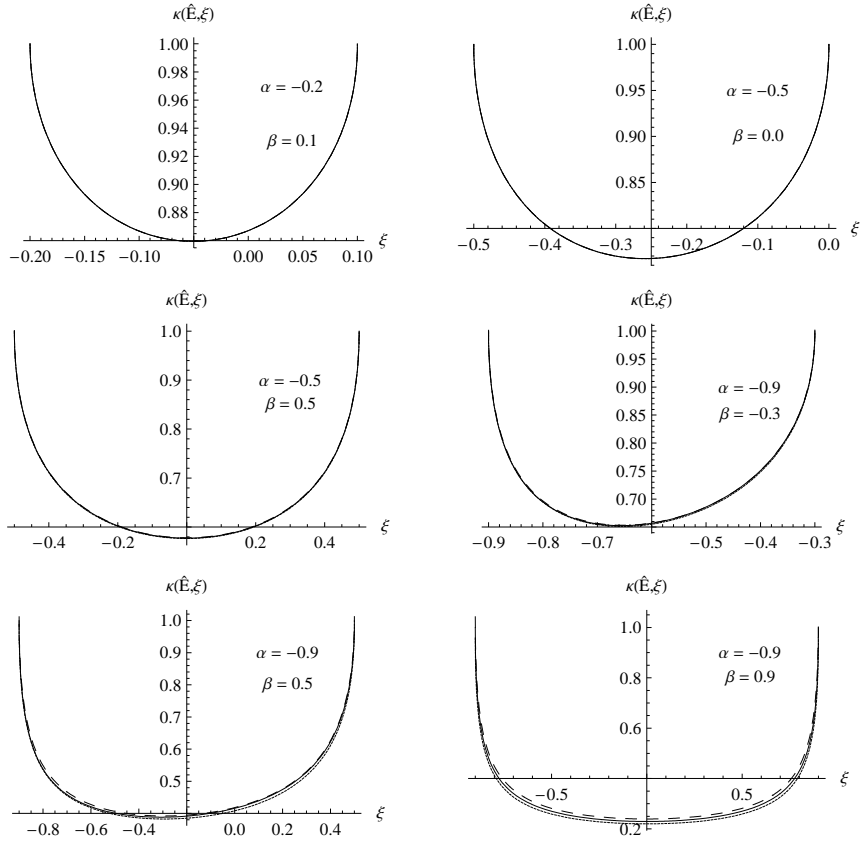}
\caption{\label{Fig_AsymptoticFactor} Plots of the graph of $\kappa(\hat{E},\xi)$ (solid line),  the graph of the upper bound in \eqref{ineq-kappa} (dashed line), and the graph of the lower bound in \eqref{ineq-kappa} (dotted line) for several values of $\alpha$ and $\beta$ and $\alpha\leq\xi\leq\beta$.}
\end{center}
\end{figure}

\section{An Extremum Problem}


In this section, we completely solve the following problem: given the length of the two intervals, say $\ell_1$ and $\ell_2$, and given the length of the gap between the two intervals, say $\ell_3$, for which set of two intervals $E=[a_1,a_2]\cup[a_3,a_4]$ with $a_2-a_1=\ell_1$, $a_4-a_3=\ell_2$ and $a_3-a_2=\ell_3$, the convergence factor $\kappa(E,0)$ is minimal?


For the linear transformed problem (see Section\,2), this problem reads as follows. Given $\hat{E}:=[-1,\alpha]\cup[\beta,1]$, $-1<\alpha<\beta<1$, for which $\xi\in(\alpha,\beta)$ the convergence factor $\kappa(\hat{E},\xi)$ is minimal? The answer gives the following theorem.


\begin{theorem}
Let $\hat{E}:=[-1,\alpha]\cup[\beta,1]$, $-1<\alpha<\beta<1$, and let $k\in(0,1)$ and $\rho\in(0,K)$ be given by \eqref{k} and \eqref{sn}, respectively. Then the convergence factor $\kappa(\hat{E},\xi)$, $\alpha<\xi<\beta$, is minimal for
\begin{equation}
\xi^*=\alpha+\zn(\rho)\sqrt{(1-\alpha)(1+\beta)}.
\end{equation}
\end{theorem}
\begin{proof}
Let $f(u):=\Theta(u-\rho)/\Theta(u+\rho)$. In \cite{Sch-2005}, it is proved that $f''(u)>0$, $0<u<K$, with $f(0)=f(1)=1$. By \eqref{kappa-case1}, $\kappa(\hat{E},\xi)=f(v^*)$, where $v^*$ is uniquely determined by \eqref{snv*}. By Lemma\,\ref{Lemma-dTheta},
\begin{equation}\label{dTheta=0}
f'(v^*)=0\iff\zn(\rho)\bigl[1-k^2\sn^2(v^*)\sn^2(\rho)\bigr]=k^2\sn^2(v^*)\sn(\rho)\cn(\rho)\dn(\rho).
\end{equation}
By \eqref{k}, \eqref{sn}, \eqref{cndn} and \eqref{snv*},
\[
1-k^2\sn^2(v^*)\sn^2(\rho)=\frac{1+\alpha}{1+\xi}
\]
and
\[
k^2\sn^2(v^*)\sn(\rho)\cn(\rho)\dn(\rho)=\frac{(\xi-\alpha)(1+\alpha)}{(1+\xi)\sqrt{(1-\alpha)(1+\beta)}}.
\]
Thus, by \eqref{dTheta=0},
\begin{align*}
f'(v^*)=0&\iff\frac{1+\alpha}{1+\xi}\cdot\zn(\rho)=\frac{(\xi-\alpha)(1+\alpha)}{(1+\xi)\sqrt{(1-\alpha)(1+\beta)}}\\
&\iff\xi=\alpha+\zn(\rho)\sqrt{(1-\alpha)(1+\beta)}.
\end{align*}
\end{proof}

\section{Bounds for the Logarithmic Capacity of Two Intervals}


\begin{theorem}\label{Thm_LB}
Let $\hat{E}:=[-1,\alpha]\cup[\beta,1]$, $-1<\alpha\leq\beta<1$,
then
\begin{equation}\label{LB-cap}
\CAP\hat{E}\geq\frac{1}{2}\Biggl(\frac{\sqrt[4]{1-\alpha^2}+\sqrt[4]{1-\beta^2}}
{\sqrt[4]{(1-\alpha)(1+\beta)}+\sqrt[4]{(1+\alpha)(1-\beta)}}\Biggr)^4=:C_1,
\end{equation}
where equality is attained if $\alpha=\beta$ or if $\alpha\to-1$ ($\beta$ fixed) or if $\beta\to1$ ($\alpha$ fixed).
\end{theorem}
\begin{proof}
Let $-1<\alpha<\beta<1$ be given, and let $k\in(0,1)$ and $\rho\in(0,K)$ be given by \eqref{k} and \eqref{sn}, respectively. By Theorem\,\ref{Thm-AchieserCap} and Lemma\,\ref{Lemma-IneqTheta},
\[
\CAP\hat{E}=\frac{1+\beta}{2(1+\alpha)}\cdot\frac{\Theta^4(0)}{\Theta^4(\rho)}
\geq\frac{1+\beta}{2(1+\alpha)}\Bigl(\frac{\sqrt{k'}+\dn(\rho)}{1+\sqrt{k'}}\Bigr)^4.
\]
Using \eqref{k'} and \eqref{cndn}, inequality \eqref{LB-cap} follows. Concerning the cases of equality: If $\alpha=\beta$, then, for $C_1$ in \eqref{LB-cap}, we have $C_1=1/2=\CAP[-1,1]$. Further, for fixed $\beta$, $\lim_{\alpha\to-1}C_1=(1-\beta)/4=\CAP[\beta,1]$ and, for fixed $\alpha$, $\lim_{\beta\to1}C_1=(1+\alpha)/4=\CAP[-1,\alpha]$.
\end{proof}


\begin{remark}
\begin{enumerate}
\item In \cite{Solynin}, A.Yu.\,Solynin gave an excellent lower bound for the logarithmic capacity of the union of several intervals, see also \cite{Sch-2008-1} and \cite{Sch-2008-2} for a discussion of this result. Although we could not achieve the goodness of Solynin's bound in the two interval case, we found it useful to give this very simple lower bound \eqref{LB-cap}.
\item In the recent paper \cite{DubininKarp}, Dubinin and Karp even improved Solynin's lower bound and, in addition, based on a result of Haliste\,\cite{Haliste}, they gave an upper bound for the logarithmic capacity of several intervals. For the two intervals case, the result reads as follows.
\end{enumerate}
\end{remark}


\begin{theorem}[Dubinin\,\&\,Karp\,\cite{DubininKarp}]\label{Thm-UBcap}
Let $\hat{E}:=[-1,\alpha]\cup[\beta,1]$, $-1<\alpha<\beta<1$, then
\begin{equation}\label{UB-cap}
\CAP\hat{E}\leq\tfrac{1}{4}\Bigl(\sqrt{(1+\alpha)(1+\beta)}+\sqrt{(1-\alpha)(1-\beta)}\Bigr),
\end{equation}
where equality is attained if $\alpha=\beta$ or if $\alpha=-\beta$.
\end{theorem}


\begin{remark}
\begin{enumerate}
\item Numerical computations show that the upper bound in \eqref{UB-cap} is excellent if the modulus $k$ defined in \eqref{k} is not too large. If the modulus $k$ is near to $1$, i.e., if, for fixed $\alpha$, the endpoint $\beta$ is near $1$, then the upper bound derived in \cite{Sch-2008-1} is better (i.e.\ smaller) than that of \eqref{UB-cap}.
\item In Fig.\,\ref{Fig_LogCapacity}, for $\alpha\in\{-0.8,-0.3,0.3,0.8\}$ and $\alpha\leq\beta\leq1$, we have plotted the graph of $\CAP\hat{E}$ (solid line), the graph of the lower bound \eqref{LB-cap} (dashed line), and the graph of the upper bound \eqref{UB-cap} (dotted line). As one can see, the upper bound matches nearly perfect whereas the lower bound is also quite good.
\item With the help of Lemma\,\ref{Lemma-IneqTheta} and analogously to the proof of Theorem\,\ref{Thm_LB}, it is also possible to obtain an upper bound for $\CAP\hat{E}$. Since from numerical computations it turns out that this upper bound is never better than the very simple upper bound \eqref{UB-cap}, we decided to skip it.
\end{enumerate}
\end{remark}


\begin{figure}[ht]
\begin{center}
\includegraphics[scale=0.9]{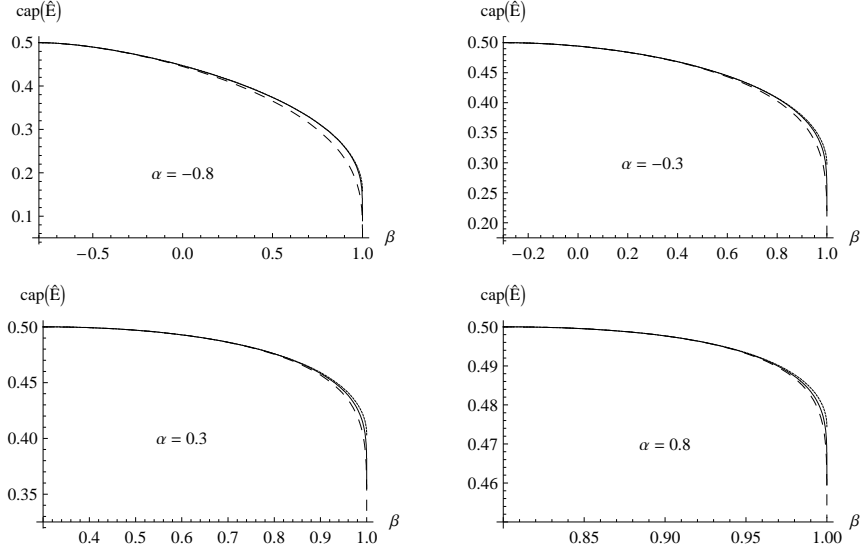}
\caption{\label{Fig_LogCapacity} Plots of the graph of $\CAP\hat{E}$ (solid line), the graph of the lower bound \eqref{LB-cap} (dashed line), and the graph of the upper bound \eqref{UB-cap} (dotted line) for several values of $\alpha$ and $\alpha\leq\beta\leq1$.}
\end{center}
\end{figure}


With the help of Theorem\,\ref{Thm-UBcap}, we get a very accurate inequality for $\Theta(u)/\Theta(0)$.


\begin{corollary}
For $0<k<1$ and $0\leq{u}\leq{K}$
\begin{equation}\label{IneqTheta-1}
\frac{\Theta^4(u)}{\Theta^4(0)}\geq\frac{1}{\dn(u)(\cn^2(u)+k'\sn^2(u))},
\end{equation}
where equality is attained if $u=0$ or if $u=\frac{1}{2}K$ or if
$u=K$ or if $k\to0$.
\end{corollary}
\begin{proof}
Let $-1<\alpha<\beta<1$ be fixed and let $k\in(0,1)$ and $\rho\in(0,K)$ be given by \eqref{k} and \eqref{sn}. By \eqref{k}, \eqref{sn}, and \eqref{cndn},
\[
\frac{1}{4}\Bigl(\sqrt{(1+\alpha)(1+\beta)}+\sqrt{(1-\alpha)(1-\beta)}\Bigr)=\frac{\cn^2(\rho)+k'\sn^2(\rho)}{2\,\dn(\rho)}
\]
which together with \eqref{cndn}, Theorem\,\ref{Thm-AchieserCap} and Theorem\,\ref{Thm-UBcap} gives
\[
\CAP\hat{E}=\frac{1}{2\dn^2(\rho)}\cdot\frac{\Theta^4(0)}{\Theta^4(\rho)}\leq\frac{\cn^2(\rho)+k'\sn^2(\rho)}{2\,\dn(\rho)}.
\]
The cases of equality follow immediately from \eqref{sncndn-K}, Lemma\,\ref{Lemma-ThetaK2} and Lemma\,\ref{Lemma-Theta}.
\end{proof}

\section{Auxiliary Results for Jacobi's Elliptic and Theta Functions}


Let $k$, $0<k<1$, be the modulus of Jacobi's elliptic functions $\sn(u)\equiv\sn(u,k)$, $\cn(u)\equiv\cn(u,k)$, and $\dn(u)\equiv\dn(u,k)$, of Jacobi's theta functions $\Theta(u)\equiv\Theta(u,k)$, $H(u)\equiv{H}(u,k)$, $H_1(u)\equiv{H}_1(u,k)$, and $\Theta_1(u)\equiv\Theta_1(u,k)$, (Jacobi's old notation) and, finally, of Jacobi's zeta function, $\zn(u)\equiv\zn(u,k)$. Here we follow the notation of Carlson and Todd~\cite{CarlsonTodd}, in other references, like \cite{Lawden}, Jacobi's zeta function is denoted by $Z(u)$.


Let $k':=\sqrt{1-k^2}$ be the complementary modulus, let $K\equiv{K}(k)$ be the complete elliptic integral of the first kind and let $K'\equiv{K}'(k):=K(k')$. Note that $K,K'\in\R^+$. Further let $q\equiv{q}(k):=\exp(-\pi{K'}/K)$ be the nome of Jacobi's theta functions.


For the definitions and many important properties of Jacobi's elliptic and theta functions, we refer to \cite{BF}, \cite{Lawden} and \cite{AS}.


Let us mention that there is a different notation of the four theta functions (e.g.\ in \cite{BF} and \cite{Lawden}) given by $\Theta(u,k)=\theta_0(v,q)=\theta_4(v,q)$, $H(u,k)=\theta_1(v,q)$, $H_1(u,k)=\theta_2(v,q)$ and $\Theta_1(u,k)=\theta_3(v,q)$, where instead of the parameter $k$ the parameter $q$ is used and $v=u\pi/(2K)$. Sometimes also the parameter $\tau=\ii{K}'/K$ is used.


The main issue of this section is to derive an upper and a lower bound for the theta function $\Theta(u)$ in terms of Jacobi's elliptic function $\dn(u)$ and the modulus $k$, see Lemma\,\ref{Lemma-IneqTheta}. For this reason, we have to prove a sequence of several lemmas.


Let us start by repeating some useful formulae. By \cite[Eq.\,(121.00)]{BF},
\begin{equation}\label{sncndn}
\sn^2(u)+\cn^2(u)=1,\qquad k^2\sn^2(u)+\dn^2(u)=1,
\end{equation}
and, by \cite[Eq.\,(1052.02)]{BF},
\begin{equation}\label{H-H1-T1}
H(u)=\sqrt{k}\,\sn(u)\,\Theta(u),\,H_1(u)=\tfrac{\sqrt{k}}{\sqrt{k'}}\,\cn(u)\,\Theta(u),\,
\Theta_1(u)=\tfrac{1}{\sqrt{k'}}\,\dn(u)\,\Theta(u),
\end{equation}
and, by \cite[Eq.\,(122.10)]{BF} and \cite[Eq.\,(3.6.2)]{Lawden},
\begin{equation}\label{sncndn-K}
\begin{aligned}
&\sn(0)=\zn(0)=0,\quad\cn(0)=1,\quad\dn(0)=1,\\
&\sn(K)=1,\quad\cn(K)=\zn(K)=0,\quad\dn(K)=k',\\
&\sn(\tfrac{1}{2}K)=\tfrac{1}{\sqrt{1+k'}},\,\cn(\tfrac{1}{2}K)=\sqrt{\tfrac{k'}{1+k'}},
\,\dn(\tfrac{1}{2}K)=\sqrt{k'},\,\zn(\tfrac{1}{2}K)=\tfrac{1}{2}(1-k),
\end{aligned}
\end{equation}
Further, by \cite[Eq.\,(731.01)--(731.03)]{BF} and \cite[Eqs.\,(3.4.25)~and~(3.6.1)]{Lawden},
\begin{equation}\label{d-sncndnzn}
\begin{aligned}
\frac{\partial}{\partial{u}}\{\sn(u)\}=\cn(u)\dn(u),\qquad
\frac{\partial}{\partial{u}}\{\cn(u)\}=-\sn(u)\dn(u),\\
\frac{\partial}{\partial{u}}\{\dn(u)\}=-k^2\sn(u)\cn(u),\qquad
\frac{\partial}{\partial{u}}\{\zn(u)\}=\dn^2(u)-E/K,
\end{aligned}
\end{equation}
and, by \cite[Lem.\,4]{Sch-2005},
\begin{equation}\label{d-Theta}
\begin{aligned}
\frac{\partial}{\partial{u}}\{\Theta(u)\}=\Theta(u)\zn(u),\quad
\frac{\partial}{\partial{u}}\{\Theta_1(u)\}=\tfrac{1}{\sqrt{k'}}\,\Theta(u)\bigl(-k^2\sn(u)\cn(u)+\dn(u)\zn(u)\bigr).
\end{aligned}
\end{equation}


Next, let us collect some basic properties of Jacobi's theta function $\Theta(u)$ in the following lemma.


\begin{lemma}\label{Lemma-Theta}
The function $\Theta(u)$ has the following properties:
\begin{enumerate}
\item $\Theta(u)>0$ for $u\in\R$ and $\Theta(u+2K)=\Theta(u)$ for $u\in\C$.
\item $\Theta(u)$ is strictly monotone increasing in $[0,K]$ and strictly monotone decreasing in $[K,2K]$.
\item $\Theta(0)\leq\Theta(u)\leq\Theta(K)$ for $u\in\R$.
\item $\Theta(0)=\Theta_1(K)=\sqrt{k'}\,\Theta(K)=\sqrt{k'}\,\Theta_1(0)=\sqrt{2k'K/\pi}$
\item For $k\to0$ there is $\Theta(u)\to1$, $u\in\C$.
\end{enumerate}
\end{lemma}


For the next lemma, see Lemma\,2 of \cite{Sch-2008-1}. Unfortunately, there is a misprint in the formula of $H(\tfrac{1}{2}K)$, which is here corrected.


\begin{lemma}\label{Lemma-ThetaK2}
Let $0<k<1$, then
\begin{equation}
\begin{aligned}
\Theta^4(\tfrac{1}{2}K)&=\Theta_1^4(\tfrac{1}{2}K)=\tfrac{2}{\pi^2}(1+k')\sqrt{k'}K^2,\\
H^4(\tfrac{1}{2}K)&=H_1^4(\tfrac{1}{2}K)=\tfrac{2}{\pi^2}(1-k')\sqrt{k'}K^2.
\end{aligned}
\end{equation}
\end{lemma}


\begin{lemma}\label{Lemma-zn}
The function
\begin{equation}
f(u):=\zn(u)-\frac{k^2\sn(u)\cn(u)}{\sqrt{k'}+\dn(u)}
\end{equation}
has the following properties:
\begin{enumerate}
\item $f(0)=f(\frac{1}{2}K)=f(K)=0$
\item $f(u)<0$ for $0<u<\frac{1}{2}K$ and $f(u)>0$ for $\frac{1}{2}K<u<K$
\item $f''(0)=f''(\frac{1}{2}K)=f''(K)=0$
\item $f''(u)>0$ for $0<u<\frac{1}{2}K$ and $f''(u)<0$ for $\frac{1}{2}K<u<K$
\end{enumerate}
\end{lemma}
\begin{proof}
(i) follows immediately from \eqref{sncndn-K}. Let us prove (iii) and (iv), from which (ii) follows. Computing and simplifying $f''(u)$ with the help of \eqref{d-sncndnzn} and \eqref{sncndn} leads to
\[
f''(u)=\frac{\sqrt{k'}(1+k')(1-k')^3(\dn(u)-\sqrt{k'})\,\sn(u)\,\cn(u)}{(\sqrt{k'}+\dn(u))^3},
\]
thus, by \eqref{sncndn-K}, (iii) follows. Since $\dn(\frac{1}{2}K)=\sqrt{k'}$ and $\dn(u)$ is strictly monotone decreasing in $u$, $0\leq{u}\leq{K}$, and since $\sn(u)>0$ and $\cn(u)>0$ for $0<u<K$, assertion\,(iv) follow.
\end{proof}


\begin{lemma}\label{Lemma-Theta+Theta1}
The function
\begin{equation}
f(u):=\Theta(u)+\Theta_1(u)
\end{equation}
is strictly monotone decreasing on $[0,\frac{1}{2}K]$ and strictly monotone increasing on $[\frac{1}{2}K,K]$. Moreover,
$f(u+K)=f(u)$.
\end{lemma}
\begin{proof}
By \eqref{d-Theta},
\[
f'(u)=\frac{1}{\sqrt{k'}}\,\Theta(u)\bigl(\sqrt{k'}+\dn(u)\bigr)\Bigl(\zn(u)-\frac{k^2\sn(u)\,\cn(u)}{\sqrt{k'}+\dn(u)}\Bigr).
\]
By Lemma\,\ref{Lemma-Theta}\,(i) and Lemma\,\ref{Lemma-zn}, we get $f'(u)<0$ for $0<u<\frac{1}{2}K$ and $f'(u)>0$ for $\frac{1}{2}K<u<K$. Since $\Theta(u+K)=\Theta_1(u)$ and $\Theta_1(u+K)=\Theta(u)$, the second relation follows.
\end{proof}


\begin{lemma}\label{Lemma-IneqTheta}
For $u\in\R$,
\begin{equation}
\sqrt[4]{8(1+k')}\sqrt[8]{k'}\leq\frac{\Theta(u)}{\Theta(0)}\,\bigl(\sqrt{k'}+\dn(u)\bigr)\leq1+\sqrt{k'},
\end{equation}
where equality is attained in both inequalities for $k\to0$, in the left inequality for $u=(\nu+\frac{1}{2})K$, $\nu\in{\mathbb{Z}}$, and in the right inequality for $u=\nu{K}$, $\nu\in{\mathbb{Z}}$.
\end{lemma}
\begin{proof}
By Lemma\,\ref{Lemma-Theta+Theta1},
\[
\Theta(\tfrac{1}{2}K)+\Theta_1(\tfrac{1}{2}K)\leq\Theta(u)+\Theta_1(u)\leq\Theta(0)+\Theta_1(0)
\]
which, by \eqref{H-H1-T1}, Lemma\,\ref{Lemma-ThetaK2} and Lemma\,\ref{Lemma-Theta}, is equivalent to
\[
\frac{\sqrt[4]{8(1+k')}\sqrt[8]{k'}}{\sqrt{k'}}\,\Theta(0)\leq\Theta(u)\Bigl(1+\frac{\dn(u)}{\sqrt{k'}}\Bigr)
\leq\frac{1+\sqrt{k'}}{\sqrt{k'}}\,\Theta(0).
\]
The cases of equality follow immediately from \eqref{sncndn-K}, Lemma\,\ref{Lemma-Theta} and Lemma\,\ref{Lemma-ThetaK2}.
\end{proof}


\begin{lemma}\label{Lemma-dTheta}
Let $a\in\C$ be fixed. Then
\begin{equation}\label{dTheta}
\frac{\partial}{\partial{u}}\Bigl\{\frac{\Theta(u-a)}{\Theta(u+a)}\Bigr\}
=-\frac{\Theta(u-a)}{\Theta(u+a)}\Bigl[2\,\zn(a)-\frac{2k^2\sn^2(u)\sn(a)\cn(a)\dn(a)}{1-k^2\sn^2(u)\sn^2(a)}\Bigr].
\end{equation}
\end{lemma}
\begin{proof}
Using \eqref{d-Theta}, we get (where $\Theta'(u):=\frac{\partial}{\partial{u}}\{\Theta(u)\}$)
\begin{align*}
\frac{\partial}{\partial{u}}\Bigl\{\frac{\Theta(u-a)}{\Theta(u+a)}\Bigr\}
&=\frac{\Theta'(u-a)}{\Theta(u+a)}\cdot\frac{\Theta(u-a)}{\Theta(u-a)}-\frac{\Theta(u-a)\,\Theta'(u+a)}{\Theta^2(u+a)}\\
&=\frac{\Theta(u-a)}{\Theta(u+a)}\Bigl[\frac{\Theta'(u-a)}{\Theta(u-a)}-\frac{\Theta'(u+a)}{\Theta(u+a)}\Bigr]\\
&=-\frac{\Theta(u-a)}{\Theta(u+a)}\bigl[\zn(u+a)-\zn(u-a)\bigr],
\end{align*}
thus Eq.\,\eqref{dTheta} follows immediately by the formulae \cite[Eq.\,(3.6.2)]{Lawden}
\[
\zn(u\pm{a})=\zn(u)\pm\zn(a)\mp{k}^2\sn(u)\sn(a)\sn(u\pm{a})
\]
and by the formula for $\sn(u\pm{a})$, see \cite[Eq.\,(123.01)]{BF}.
\end{proof}


\bibliographystyle{amsplain}

\bibliography{AsymptoticFactor}

\end{document}